\documentclass[a4paper, 12pt]{amsart}
\usepackage{geometry}
\usepackage{amsthm,amssymb,amsmath,hyperref, array}
\usepackage{color, mdframed}
\usepackage{tikz}
\usetikzlibrary{decorations.pathreplacing}
\usepackage[utf8]{inputenc}
\usepackage{graphicx,enumerate}
\usepackage{mathtools,bm}
\usepackage{enumitem, booktabs}

\makeatletter
\def\l@section{\@tocline{1}{12pt plus2pt}{0pt}{}{\bfseries}}
\def\l@subsection{\@tocline{2}{0pt}{2pc}{2pc}{}}
\makeatother
\makeatletter
\def\subsection{\@startsection{subsection}{2}{\z@}%
	{-3.25ex\@plus -1ex \@minus -.2ex}%
	{1.5ex \@plus .2ex}%
	{\normalfont\bfseries\boldmath}}
\def\subsubsection{\@startsection{subsubsection}{3}%
	\z@{.5\linespacing\@plus.7\linespacing}{-.5em}%
	{\normalfont\bfseries\boldmath}}
\renewcommand\paragraph{\@startsection{paragraph}{4}{\z@}%
	{3.25ex \@plus1ex \@minus.2ex}%
	{-1em}%
	{\normalfont\normalsize\bfseries}}
\makeatother

\theoremstyle{plain}
\newtheorem{thm}{Theorem}[section]

\newtheorem{lem}[thm]{Lemma}

\theoremstyle{definition}

\theoremstyle{remark}
\newtheorem{rem}[thm]{Remark}

\theoremstyle{plain}
\newtheorem{conj}[thm]{Conjecture}
\numberwithin{equation}{section}

\newtheorem{opques}[thm]{Open Question}

\theoremstyle{plain} 
\newcommand{\thistheoremname}{}
\newtheorem{genericthm}[thm]{\thistheoremname}
\newenvironment{namedthm}[1]
  {\renewcommand{\thistheoremname}{#1}%
   \begin{genericthm}}
  {\end{genericthm}}

  \newtheorem*{genericthm*}{\thistheoremname}
\newenvironment{namedthm*}[1]
  {\renewcommand{\thistheoremname}{#1}%
   \begin{genericthm*}}
  {\end{genericthm*}}

\newcommand{\del}{\delta}

\newcommand{\D}{{\mathbb D}}

\newcommand{\calB}{{\mathcal B}}

\newcommand{\calD}{{\mathcal D}}

\newcommand{\calM}{{\mathcal M }}
\newcommand{\calN}{{\mathcal N}}

\newcommand{\calQ}{{\mathcal Q}}

\newcommand{\calT}{{\mathcal T}}

\makeatletter
\newcommand{\vast}{\bBigg@{4}}
\newcommand{\Vast}{\bBigg@{5}}
\makeatother

\newcommand{\frakT}{{\mathfrak T}}

\def\udot#1{\ifmmode\oalign{$#1$\crcr\hidewidth.\hidewidth
    }\else\oalign{#1\crcr\hidewidth.\hidewidth}\fi}

\def\beq{\begin{equation}}
\def\eeq{\end{equation}}
\makeatletter
\newcommand{\doublewidetilde}[1]{{%
  \mathpalette\double@widetilde{#1}%
}}
\newcommand{\double@widetilde}[2]{%
  \sbox\z@{$\m@th#1\widetilde{#2}$}%
  \ht\z@=.9\ht\z@
  \widetilde{\box\z@}%
}
\makeatother

\def\one{\mbox{1\hspace{-4.25pt}\fontsize{12}{14.4}\selectfont\textrm{1}}}

\usepackage{tikz}

\makeatletter
\def\@makefnmark{%
  \leavevmode
  \raise.9ex\hbox{\fontsize\sf@size\z@\normalfont\tiny\@thefnmark}}
\makeatother

\begin{document}
	
\title[Near-endpoints Carleson embeddings]{Near-endpoints Carleson Embedding of $\calQ_s$ and $F(p, q, s)$ into tent spaces}

\author{Bingyang Hu}
\address{Department of Mathematics and Statistics\\
         Auburn University\\
         Auburn, Alabama, U.S.A, 36849}
\email{bzh0108@auburn.edu}

\author{Xiaojing Zhou}
\address{Department of Mathematics\\
         Shantou University\\
         Shantou, Guangdong, China, 515063}
\email{xjzhou2023@163.com}

\begin{abstract}
This paper aims to study the $\calQ_s$ and $F(p, q, s)$ Carleson embedding problems near endpoints. We first show that for $0<t<s \le 1$, $\mu$ is an $s$-Carleson measure if and only if $id: \calQ_t \mapsto \calT_{s, 2}^2(\mu)$ is bounded. Using the same idea, we also prove a near-endpoints Carleson embedding for $F(p, p\alpha-2, s)$ for $\alpha>1$. Our method is different from the previously known approach which involves a delicate study of Carleson measures (or logarithmic Carleson measures) on weighted Dirichlet spaces.
As some byproducts, the corresponding compactness results are also achieved. Finally, we compare our approach with the existing solutions of Carleson embedding problems proposed by Xiao, Pau, Zhao, Zhu, etc.  

Our results assert that a ``tiny-perturbed" version of a conjecture on the $\calQ_s$ Carleson embedding problem due to Liu, Lou, and Zhu is true. Moreover, we answer an open question by Pau and Zhao on the $F(p, q, s)$ Carleson embedding near endpoints.

\end{abstract}
\date{\today}

\thanks{}

\maketitle


\section{Introduction} \label{20240627sec01}
The current paper is motivated by the study of the embedding problems between $\calQ_s$ spaces and tent spaces via Carleson measures. More precisely, we are interested in the following conjecture. 

\begin{conj}[{\cite{LLZ2017}}] \label{20240614conj01}
    For any $0<s<1$ and $\mu$ being a positive Borel measure on the open unit disc $\D$, $id: \calQ_s \mapsto \calT_{s, 2}^2(\mu)$ is bounded if and only if $\mu$ is an $s$-Carleson measure.
\end{conj}

Our first main result is the following, which asserts that the above conjecture holds under a small perturbation. 

\begin{thm} \label{20240613thm02}
For any $0<t<s \le 1$ and $\mu$ being a positive Borel measure on $\D$, $id: \calQ_t \mapsto \calT_{s, 2}^2(\mu)$ is bounded if and only if $\mu$ is an $s$-Carleson measure. 
\end{thm}

More generally, we have the following near-endpoints result on the $F(p, p-2, s)$ Carleson embedding. 

\begin{thm} \label{20240624thm01}
For any $p>1$, $0<t<s \le 1$ and $\mu$ being a positive Borel measure on $\D$, $id: F(p, p-2, t) \mapsto \calT_{s, p}^p(\mu)$ is bounded if and only if $\mu$ is an $s$-Carleson measure. 
\end{thm}

Here, for any $p>0, \alpha, s>0$, $\beta>-1$ and $\mu$ being a positive Borel measure on $\D$, 
\begin{enumerate}
\item [$\bullet$] $\calB_\alpha$ refers to the \emph{$\alpha$-Bloch space} which consists of all the holomorphic functions on $\D$ such that 
$$
\left\|f \right\|_{\calB_\alpha}:=|f(0)|+\sup\limits_{z \in \D} |f'(z)|(1-|z|^2)^\alpha<+\infty.
$$
In particular, when $\alpha=1$, $\calB_1=\calB$ is the classical \emph{Bloch space}.
\item [$\bullet$] $\calD_\beta^p$ is the \emph{weighted Dirichlet space} which consists of all holomorhic functions $f$ on $\D$ such that
$$
\left\|f \right\|_{\calD_\beta^p}:=|f(0)|+\left( \int_{\D} |f'(z)|^p (1-|z|^2)^\beta dA(z) \right)^{\frac{1}{p}},
$$
where $dA$ is the standard Lebesgue measure on $\D$. 
\item [$\bullet$] $\calQ_s$ is the collection of all holomorphic functions on $\D$ such that
$$
\left\|f \right\|_{\calQ_s}:=|f(0)|+\sup_{a\in\D} \sqrt{ \int_{\D} |f'(z)|^2(1-|\varphi_{a}(z)|^2)^s dA(z)}<+\infty,
$$
where $\varphi_{a}(z):=\frac{a-z}{1-\overline{a}z}$ is the disc automorphism with base $a \in \D$. It is well-known that $\calQ_1=\textnormal{BMOA}$ and $\calQ_s=\calB$ when $s>1$ (see, e.g. \cite{Xiaojie2001}).

\medskip 

\item [$\bullet$] $F(p, p\alpha-2, s)$ is defined as the space of all holomorphic functions $f$ on $\D$ with satisfying 
\begin{eqnarray*}
&&\left\|f \right\|_{F(p, p\alpha-2, s)}\\
&& \qquad :=|f(0)| +\sup_{a \in \D} \left( \int_{\D} |f'(z)|^p \left(1-|z|^2 \right)^{p\alpha-2} (1-|\varphi_a(z)|^2)^s dA(z) \right)^{\frac{1}{p}}<+\infty. 
\end{eqnarray*}
Observe that $\calQ_s=F(2, 0, s)$. Moreover, it is also known that $F(p, p\alpha-2, s)=\calB^\alpha$ when $s>1$ (see, e.g., \cite[Theorem 1.3]{Zhao1996}). 

\medskip 

\item [$\bullet$] $\calT_{s, p}^p(\mu)$ consists of all holomorphic functions on $\D$ such that 
\begin{equation} \label{20240614cond01}
\left\|f \right\|_{\calT_{s, p}^p}:=\sup_{I \subset \partial \D} \left(\frac{1}{|I|^s \left(\log{\frac{2}{|I|}}\right)^p}\int_{S(I)}|f(z)|^pd\mu(z) \right)^{\frac{1}{p}}<+\infty,
\end{equation} 
where $S(I):=\left\{z\in\D:1-|I|<|z|<1, \ \frac{z}{|z|} \in I\right\}$ is the \emph{Carleson tent} associated to the subarc $I \subseteq \partial \D$. 

\medskip 

\item [$\bullet$] finally, we say $\mu$ is an \emph{$s$-Carleson measure} if 
\begin{equation} \label{20240613eq01}
\left\|\mu \right\|_{{\mathcal{CM}_s}}:=\sup_{I \subset \partial \D} \frac{\mu(S(I))}{|I|^s}<+\infty. 
\end{equation} 
\end{enumerate}

\begin{rem} \label{20240624rem01}
If $0<p \le 1$, Theorem \ref{20240624thm01} indeed holds at the endpoint $t=s$. More precisely, we can show that \emph{ for any $0<p, s \le 1$ with\footnote{Otherwise, the $F(p, p-2, s)$ space contains only constant functions (see, \cite[Proposition 2.12]{Zhao1996}.} $p+s>1$ and $\mu$ being a positive Borel measure on $\D$, $id: F(p, p-2, s) \mapsto \calT_{s, p}^p(\mu)$ is bounded if and only if $\mu$ is an $s$-Carleson measure}. The proof of this assertion follows from an easy modification of the argument in \cite[Lemma 3.2 and Theorem 3.1]{PauZhao2014}, and hence we omit it here. The key observation is to note that if $0<p \le 1$ and $\mu$ is an $s$-Carleson measure, then the embedding $id: \calD_{s+p-2}^2 \mapsto L^p(\mu)$ is bounded (see, the proof of \cite[Lemma 3.2]{PauZhao2014}). Therefore, we only focus on the case when $p>1$ in Theorem \ref{20240624thm01}.
\end{rem}

\medskip 

In the second part of this paper, we consider the more general $F(p, p\alpha-2, s)$ Carleson embedding problems. Our goal is to understand the following open question due to Pau and Zhao \cite{PauZhao2014}. 

\begin{opques}[\cite{PauZhao2014}] \label{20240624ques01}
What is the criterion for the boundedness and compactness of the embedding $id: F(p, p\alpha-2, s) \mapsto \frakT_{s, p}^{\infty}(\mu)$?
\end{opques}
Here, $\frakT_{s, p}^{\infty}(\mu)$ is the tent space which consists of all holomorphic functions $f$ on $\D$ such that 
\begin{equation} \label{20240616eq41}
\sup\limits_{I \subseteq \partial \D} \frac{1}{|I|^s} \int_{S(I)} |f(z)|^p d\mu(z)<+\infty.
\end{equation} 

Our second main theorem is the following, which answers Question \ref{20240624ques01} near endpoints. 

\begin{thm} \label{20240624thm10}
Let $\alpha>1, s>t>0$ and $p>1$. Then $id: F(p, p\alpha-2, t) \to \frakT_{s, p}^{\infty}(\mu)$ is bounded if and only if $\mu$ is an $\left[s+p(\alpha-1)\right]$-Carleson measure. 
\end{thm} 

\begin{rem} \label{20240626rem10}
(1) We first point out that in Theorem \ref{20240624thm10}, it suffices to consider the case when both $\alpha, p>1$. Indeed, 
\begin{enumerate}
\item [(a)] the $F(p, p\alpha-2, s)$ embedding problem has a simple solution when $0<\alpha<1$ (see, \cite[Theorem 4.2]{PauZhao2014}). This is due to the fact that $\calB^\alpha \subseteq H^\infty(\D)$ for $0<\alpha<1$. Here, $H^\infty(\D)$ is the disc algebra. 

\item [(b)] Theorem \ref{20240624thm01} can be interpreted as the endpoint version of Theorem \ref{20240624thm10} at $\alpha=1$. However, Theorem \ref{20240624thm10} does \emph{not} contain Theorem \ref{20240624thm01} since one has to introduce the logarithmic term $\left(\log \frac{2}{|I|} \right)^p$ in the definition of the tent spaces when $\alpha=1$ (see, \eqref{20240614cond01}). Therefore, in Theorem \ref{20240624thm10}, we only consider the case when $\alpha>1$.

\item [(c)]  similarly as in Remark \ref{20240624rem01}, when $0<p \le 1$ with satisfying $s+p\alpha>1$, Theorem \ref{20240624thm10} is true at the endpoint $t=s$ (see, \cite[Theorem 4.4]{PauZhao2014}). This is because again the embedding $id: \calD_{s+p\alpha-2}^p \mapsto L^p(\mu)$ is bounded when $0<p \le 1$. Hence, we will only focus on the case when $p>1$ in Theorem \ref{20240624thm10}.
\end{enumerate}
(2) The compactness results corresponding to Theorems \ref{20240613thm02}, \ref{20240624thm01} and \ref{20240624thm10} are routine once the boundedness results are achieved (see, Theorems \ref{20240621thm01} and \ref{20240626thm01}). 

\end{rem}

The $\calQ_s$ and $F(p, p\alpha-2, s)$ embedding problem is an important and fruitful subject in complex function theory. Here, we would like to include a brief historical remark on this line of research\footnote{We will focus on the case when both $p, \alpha>1$.}. 
\begin{enumerate}
\item  If $s \ge 1$, then Conjecture \ref{20240614conj01} is \emph{true} (see, \cite[Theorems 1,2]{LLZ2017}). The main reason that the argument in \cite{LLZ2017} works is due to the fact that when $s \ge 1$, $\mu$ is an $s$-Carleosn measure if and only if the embedding $\calD^2_s \hookrightarrow L^2(\mu)$ is continous  (see, \cite[Lemma 2.1]{Xiaojie2008}, and also \cite{Carleson1962, Hastings1975, Stegenga1980}). However, such an equivalence is \emph{not} available when $s \in (0, 1)$ (see, \cite{Arcozzi2002}), and hence the approach in \cite{LLZ2017} does \emph{not} generalize when $s \in (0, 1)$. Similarly, in the case of $F(p, p\alpha-2, s)$ Carleson embedding for $p>1$, Pau and Zhao \cite[Theorem 4.4]{PauZhao2014}  proved that Theorem \ref{20240624thm10} holds for $t=s$, however, under an additional assumption $s+p(\alpha-1) \ge 1$. This is again because, under such an assumption, the embedding $\calD_{s+p\alpha-2}^p \hookrightarrow L^p(\mu)$ is continuous. 

In general, this line of investigation is closely related to the study of continuous embedding properties of weighted Dirichlet spaces $\calD_\beta^p \hookrightarrow L^q(\mu)$, which is understood well when $p<q$, however, was only partially studied when $p \ge q$ (see, e.g., \cite{Girela2006} and the reference therein). 

\vspace{0.05in}

\item The $\calQ_s$ embedding problem has several beautiful solutions if one allows considering different Carleson measures and tent spaces. One remarkable result is due to Xiao in \cite{Xiaojie2008}, in which he showed that $id: \calQ_s \to \frakT_{s, 2}^\infty(\mu)$ is bounded if and only if $\mu$ is a \emph{logarithmic $s$-Carleson measure}, that is,
\begin{equation} \label{20240616eq40}
\left\|\mu \right\|_{\mathcal{LCM}_s}:=\sup_{I \subseteq \partial D} \left(\log \frac{2}{|I|} \right)^2 \frac{\mu(S(I))}{|I|^s}<+\infty.
\end{equation} 

We also refer the reader \cite[Theorem 3.1]{PauZhao2014} for the generalization of the above result to the $F(p, p-2, s)$-spaces. Here, we would like to point out that the appearance of the term $\left(\log \frac{2}{|I|} \right)^2$ in the definition of the logarithmic $s$-Carleson measure is pivotal, which consequently provides a subtle cancellation of the term $\frac{1}{|1-\overline{z} w|^2}$ and hence gives the continuous embedding $\calD_s^2 \hookrightarrow L^2(\mu)$ for $0<s<1$ in this case (see, e.g., \cite[Lemma 3.2]{PauZhao2014}). 

To this end, we remark that using the idea of Xiao, Pau, and Zhao, one could formulate a Carleson embedding for the $F(p, p\alpha-2, s)$ spaces for $\alpha>1$. Observe that in this case, one shall replace the test function $f_a(z)=\log \frac{1}{1-\overline{a}z}$ (when $\alpha=1$) by $f_a(z)=\frac{1}{\left(1-\overline{a}z \right)^{\alpha-1}}$ (when $\alpha>1$). Therefore, the tent space in Xiao-Pau-Zhao's analog no longer contains a logarithmic term. We will see that in this case, Xiao-Pau-Zhao's analog is indeed a consequence of Theorem \ref{20240624thm10} (see, Figure \ref{Fig1} and Statement \ref{20240627stat01}). 

\vspace{0.05in}

\item  Very recently, Lv \cite{LV2014} provided another way to consider the $\calQ_s$ embedding problems. The idea in her approach is to define both Carleson measures and tent spaces \emph{not} over the whole Carleson tent $S(I)$ but just over the \emph{upper-half Carleson tent}\footnote{In the original presentation in \cite{LV2014}, the Carleson measures and tent spaces are indeed defined over Bergman discs with fixed radius. This is equivalent to making the definition over all the upper-half Carleson tents $S_{\textnormal{up}}(I)$.}  $S_{\textnormal{up}}(I):=\left\{z \in \D: 1-|I|<|z|<1-\frac{|I|}{2}, \ \frac{z}{|z|} \in I \right\}$. 

The point is that in this situation, there is only one scale, that is $\frac{1}{|I|}$. This is different from those Carleson embedding problems concerning the full Carleson box $S(I)$, where one has to handle multiple scales of the form $\frac{2}{1-|z|^2}$ for $z \in S(I)$. Indeed, it is more reasonable\footnote{This is because, on one hand side, a Bloch Carleson embedding always implies a $\calQ_s$ Carleson embedding as $\calQ_s \subseteq \calB$; while on the other hand side, both $\calB$ and $\calQ_s$ share the same test function $f_a(z)=\log \frac{2}{1-\overline{a}z}$. Therefore, in \cite{LV2014}, a Bloch Carleson embedding is equivalent to a $\calQ_s$ Carleson embedding.} to state the results in \cite{LV2014} simply via Bloch Carleson embeddings, rather than $\calQ_s$ Carleson embeddings. 
\end{enumerate} 

\medskip 

Returning now to our near-endpoints Carleson embedding, we would like to emphasize that our method is \emph{different} from the previous approach which relies on a careful study of Carleson measures (or logarithmic Carleson measures) on weighted Dirichlet spaces. The novelty of this approach is two-fold: 
\begin{enumerate}
    \item [$\bullet$] First of all, we represent the derivative of the function $\frac{f(z)-f(a_I)}{(1-\overline{a_I}z)^s}+f(a_I)$ via an integral representation formula (see, Lemma \ref{20240614lem01}), rather than $f(z)$ itself\footnote{This idea originates from a handwritten note (private communication) of Jordi Pau back to 2018 (see, \cite{JordiPau2018}).}. Here, $a_I:=(1-|I|)e^{i\theta}$ is the center point of the roof of the Carleson tent $S(I)$ with $e^{i\theta}$ being the center of $I$. 
    
    \item [$\bullet$] Secondly, instead of proving a continuous embedding of the type $\calD_s^2 \hookrightarrow L^2(\mu)$, we borrow a tiny power of $\frac{1}{\left|1-\overline{a_I}w \right|}$ and apply a sophisticated integral formula from \cite{zhangLishangGuo2018} to absorb the term $\frac{1}{\left|1-\overline{z}w \right|^{2(1+\alpha-s)}}$, which gives us estimates at the scale $\frac{1}{|I|}$, rather than $\frac{1}{1-|z|}, \, z \in S(I)$ (see, \eqref{20240615eq01}). 
\end{enumerate} 

To this end, we would like to make a comparison between the near-endpoints Carleson embeddings with other known solutions to the $F(p, p\alpha-2, s)$ Carleson embedding problem. Here, we refer to the reader Section \ref{20240626sec01} for more detailed information on this part. 

Without loss of generality, let us only focus on the case when $p>2$ and $0<s<1$. Let us try to understand the following question: 

\begin{namedthm}{Question} \label{20240627ques01}
For each $\alpha>1$, what are all the possible $\widetilde{s}>0$, such that the following statement holds: $id: F(p, p\alpha-2, \widetilde{s}) \hookrightarrow \frakT_{s, p}^\infty(\mu)$ if and only if $\mu$ is an $\left[s+p(\alpha-1) \right]$-Carleson measure. 
\end{namedthm}

Recall that there are three different approaches to study the above question:
\begin{enumerate}
    \item Near-endpoints Carleson embeddings (see, Theorems \ref{20240613thm02}, \ref{20240624thm01}, and \ref{20240624thm10});
    \item Xiao-Pau-Zhao's logarithmic Carleson embeddings and its analog when $\alpha>1$ (see, \cite[Theorem 1.1]{Xiaojie2008},  \cite[Theorem 3.1]{PauZhao2014}, and Statement \ref{20240627stat01});
    \item Carleson embedding induced by the continuous embedding $\calD_{s+p\alpha-2}^p \hookrightarrow L^p(\mu)$ (see, \cite[Theorem 4.4]{PauZhao2014}). 
\end{enumerate}
We summarize the solutions for each of the above approaches to Question \ref{20240627ques01} in the following figure (see, Figure \ref{Fig1}).

\begin{center}
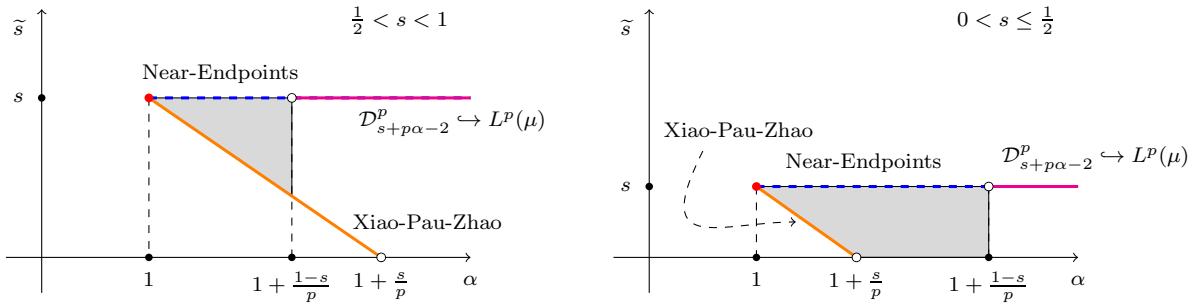
\begin{figure}[ht]
\begin{tikzpicture}[scale=4.7]
\draw (-1.7, -.1) [->] -- (-1.7,0.7);
\draw (-1.8, 0) [->] -- (-0.5,0);
\fill (-0.7, 0.6) node [above] {\tiny{$\frac{1}{2}<s<1$}};
\fill (1, 0.6) node [above] {\tiny{$0<s\le \frac{1}{2}$}};
\fill (-0.5, -0.02) node [below] {\tiny{$\alpha$}};
\fill (-1.72, 0.65) node [left] {\tiny{$\widetilde{s}$}}; 
\fill (-1.72, 0.45) node [left] {\tiny{$s$}};
\fill (-1.7, 0.45) circle [radius=.3pt]; 
\draw[-, fill=black!15, opacity=.5] (-1.4, 0.45)--(-1,0.45)--(-1, 0.173)--cycle;
\draw[-, fill=black!15, opacity=.5] (0.3, 0.2)--(0.95,0.2)--(0.95, 0)--(0.58, 0)--cycle;
\draw (0, -.1) [->] -- (0,0.7);
\draw (-0.1, 0) [->] -- (1.2,0);
\fill (1.2, -0.02) node [below] {\tiny{$\alpha$}};
\fill (-0.02, 0.65) node [left] {\tiny{$\widetilde{s}$}}; 
\fill (0, 0.2) circle [radius=.3pt];
\fill (-0.02, 0.2) node [left] {\tiny{$s$}}; 
\fill (-1.4, 0) circle [radius=.3pt];  
\fill (0.3, 0) circle [radius=.3pt]; 
\fill (-1.4,-0.01) node [below] {\tiny{$1$}};
\draw [dashed] (-1, 0)--(-1, 0.45);
\fill (-1,-0.01) node [below] {\tiny{$1+\frac{1-s}{p}$}};
\fill (-1, 0) circle [radius=.3pt]; 
\draw [opacity=1, line width=0.4mm, orange] (-1.4, 0.45)--(-0.75, 0);
\draw [opacity=1, dashed, line width=0.4mm, blue] (-1.4, 0.45)--(-0.5, 0.45);
\draw [dashed] (-1.4, 0.45)--(-1.4, 0); 
\draw [dashed] (0.3, 0.2)--(0.3, 0); 
\fill (0.3, -0.01) node [below] {\tiny{$1$}};
\fill (-0.75, -0.01) node [below] {\tiny{$1+\frac{s}{p}$}};
\draw[black] (-0.75, 0) circle (0.012); 
\draw [opacity=1, line width=0.4mm, orange] (0.3, 0.2)--(0.58, 0);
\fill [opacity=1, white] (-.75, 0) circle [radius=.3pt];
\draw [opacity=1, line width=0.4mm, magenta] (-1, 0.45)--(-0.5, 0.45);
\fill (0.58, -0.01) node [below] {\tiny{$1+\frac{s}{p}$}};
\draw[black] (0.58, 0) circle (0.012); 
\fill [opacity=1, white] (0.58, 0) circle [radius=.3pt];
\draw [opacity=1, dashed, line width=0.4mm, blue] (0.3, 0.2)--(1.2, 0.2);
\fill  [opacity=1, red] (0.3, 0.2) circle [radius=.35pt];
\fill  [opacity=1, red] (-1.4, 0.45) circle [radius=.35pt]; 
\draw[black] (-1, 0.45) circle (0.012); 
\fill [opacity=1, white] (-1, 0.45) circle [radius=.3pt];
\fill (-1.2, 0.46) node [above] {\tiny{Near-Endpoints}};
\fill (-0.86, 0.1) node [right] {\tiny{Xiao-Pau-Zhao}};
\fill (-0.55, 0.45) node [below] {\tiny{$\calD_{s+p\alpha-2}^p \hookrightarrow L^p(\mu)$}};
\fill (0.95, 0) circle [radius=.3pt];
\fill (0.95, -0.01) node [below] {\tiny{$1+\frac{1-s}{p}$}}; 
\draw [opacity=1, line width=0.4mm, magenta] (0.95, 0.2)--(1.2, 0.2);
\draw [dashed] (0.95, 0)--(0.95, 0.2); 
\draw[black] (0.95, 0.2) circle (0.012); 
\fill [opacity=1, white] (0.95, 0.2) circle [radius=.3pt];
\fill (0.6, 0.21) node [above] {\tiny{Near-Endpoints}};
\fill (1.25, 0.21) node [above] {\tiny{$\calD_{s+p\alpha-2}^p \hookrightarrow L^p(\mu)$}};
\draw [->, dashed] (0.15,0.3) ..controls (0,0) and (0.2, 0.1) .. (0.41,0.1);
\fill (0.25, 0.31) node [above] {\tiny{Xiao-Pau-Zhao}};
\end{tikzpicture}
\caption{\small{Comparison between near-endpoints Carleson embeddings with other solutions to the $F(p, p\alpha-2, s)$ Carleson embedding problem: Near-endpoints Carleson embeddings (the blue dashed line), Xiao-Pau-Zhao's logarithmic Carleson embeddings at $\alpha=1$ and its analog when $\alpha>1$ (red dot at $(1, s)$ and the orange line, respectively), Carleson embeddings induced by the continuous embedding $\calD_{s+p\alpha-2}^p \hookrightarrow L^p(\mu)$ (the magenta line), and new $(\alpha, \widetilde{s})$-ranges given by the near-endpoints Carleson embeddings (shadowed part).}}
\label{Fig1}
\end{figure}
\end{center}

\vspace{-0.2in}

The rest of the paper is organized as follows. In Section 2, we prove Theorems \ref{20240613thm02} and \ref{20240624thm01}. We also consider their corresponding compactness results. Section 3 is devoted to proving Theorem \ref{20240624thm10}. Finally, in Section 4, we compare the near-endpoints Carleson embeddings with other known solutions to the $F(p, p\alpha-2, s)$ Carleson embedding problems. 

Throughout this paper, for $a ,b \in  \mathbb{R}$, $a\lesssim b$ means there exists a positive number $C$, which is independent of $a$ and $b$, such that $a\leq C\,b$. Moreover, if both $a \lesssim  b$ and $b\lesssim a$ hold, we say $a \simeq b$.
\\
{\bf Acknowledgement.} The authors thank Zengjian Lou for introducing them to this interesting question and for sharing a handwritten draft from Jordi Pau after he visited the University of Barcelona in 2018. The authors also thank Ruhan Zhao for reading the first version of this draft carefully, as well as for his suggestions to improve the presentation of this paper. 

\bigskip 
\section{Proof of Theorems \ref{20240613thm02} and \ref{20240624thm01}}

In this section, we only focus on Theorem \ref{20240613thm02}. The proof of Theorem \ref{20240624thm01} follows from a standard modification of the proof of Theorem \ref{20240613thm02} by interchanging the role of $2$ and $p$, and hence we would like to leave the details to the interested reader. As some applications of Theorems \ref{20240613thm02} and \ref{20240624thm01}, we also study their compactness counterparts. 

\subsection{Initial reductions} The necessary part is standard. Indeed, for each $I \subseteq \partial \D$ being any subarc, consider the test function 
$$
f_{a_I}(z)=\log \frac{2}{1-\overline{a_I}{z}}. 
$$
An easy calculation yields $\sup\limits_{a_I \in \D} \left\|f_{a_I} \right\|_{\calQ_t} \lesssim 1$. 

Since $id: \calQ_t \mapsto \calT_{s, 2}^2$ is bounded, and $|I| \simeq |1-\overline{a_I} z|, z \in S(I)$, we have 
$$
\frac{\mu(S(I))}{|I|^s} \simeq  \frac{1}{|I|^s\left(\log\frac{2}{|I|}\right)^2}\int_{S(I)}\left|\log\frac{2}{1-\overline{a_I}z}\right|^2d\mu(z)\lesssim \|f_{a_I}\|^2_{\calQ_t} \lesssim 1. 
$$
The desired claim follows by taking the supremum of all $I \subseteq \partial \D$. 

\medskip 

Now we turn to prove the sufficient part, which is harder. Let $\mu$ be an $s$-Carleson measure, and for each $I \subseteq \partial D$,  we have to estimate the term 
$$
M_{I}(f):=\frac{1}{|I|^s \left(\log{\frac{2}{|I|}}\right)^2}\int_{S(I)}|f(z)|^2d\mu(z).
$$
First of all, by triangle inequality, we split the above term into two parts:
\begin{eqnarray} \label{20240614eq02}
M_I(f)%
&\lesssim& M_I^1(f)+M_I^2(f) \nonumber \\
&:=& \frac{1}{|I|^s \left(\log{\frac{2}{|I|}}\right)^2}\int_{S(I)}|f(z)-f(a_I)|^2d\mu(z)+\frac{1}{|I|^s \left(\log{\frac{2}{|I|}}\right)^2}\int_{S(I)}|f(a_I)|^2d\mu(z),
\end{eqnarray}
where $a_I$ is defined as in the argument in the necessary part. The estimate of $M_I^2(f)$ is straightforward. Indeed, 
\begin{eqnarray}  \label{20240614eq03} 
M_I^2(f)%
&=& \frac{\mu(S(I))}{|I|^s \left(\log \frac{2}{|I|} \right)^2} \cdot |f(a_I)|^2 \lesssim \frac{\mu(S(I))}{|I|^s \left(\log \frac{2}{|I|} \right)^2} \cdot \left\|f \right\|_{\calB}^2 \cdot \left(\log \frac{2}{1-|a_I|} \right)^2 \nonumber \\
&\lesssim& \left\|\mu \right\|_{\mathcal{CM}_s} \left\|f \right\|_{\calB}^2 \lesssim \left\|\mu \right\|_{\mathcal{CM}_s} \left\|f \right\|_{\calQ_t}^2,
\end{eqnarray}
where in the first estimate above, we have used \cite[(1.7)]{Xiaojie2001}. 

\medskip 

Next, we estimate $M_I^1(f)$. Using again the fact that
$$
|I| \simeq 1-|a_I|^2 \simeq \left|1-\overline{a_I}z \right|, \quad z \in S(I), 
$$
we see that 
\begin{eqnarray} \label{20240614eq04} 
M_I^1(f)%
&\simeq&   \frac{(1-|a_I|^2)^s}{\left(\log{\frac{2}{|I|}}\right)^2}\int_{S(I)}\frac{|f(z)-f(a_I)|^2}{|1-\overline{a_I}z|^{2s}}d\mu(z) \nonumber \\
&=& \frac{(1-|a_I|^2)^s}{\left(\log{\frac{2}{|I|}}\right)^2}\int_{S(I)}
    \left| \frac{f(z)-f(a_I)}{(1-\overline{a_I}z)^s}+f(a_I)-f(a_I)\right|^2d\mu(z) \nonumber  \\
    &\lesssim& M_I^{1, 1}(f)+M_I^{1, 2}(f), 
\end{eqnarray}
where 
\begin{equation} \label{20240614eq20}
M_I^{1, 1}(f):=\frac{(1-|a_I|^2)^s}{\left(\log{\frac{2}{|I|}}\right)^2}\int_{S(I)}
    \left| \frac{f(z)-f(a_I)}{(1-\overline{a_I}z)^s}+f(a_I)\right|^2d\mu(z)
\end{equation}
and 
$$
M_I^{1, 2}(f):=\frac{(1-|a_I|^2)^s}{\left(\log{\frac{2}{|I|}}\right)^2}\int_{S(I)}
    \left|f(a_I)\right|^2d\mu(z). 
$$
The term $M_I^{1, 2}(f)$ is easy to control. Indeed, 
\begin{equation}  \label{20240614eq04} 
M_I^{1, 2}(f) \lesssim \frac{(1-|a_I|^2)^s}{\left(\log{\frac{2}{|I|}}\right)^2} \|f\|_{\calB}^2 \cdot \left( \log \frac{2}{1-|a_I|} \right)^2 \mu(\D) \lesssim \left\|\mu \right\|_{\mathcal{CM}_s} \left\|f \right\|_{\calB}^2 \lesssim \left\|\mu \right\|_{\mathcal{CM}_s} \left\|f \right\|_{\calQ_t}^2. 
\end{equation} 

To estimate $M_I^{1, 1}(f)$, we need the following lemma. 

\begin{lem} \label{20240614lem01}
For $\beta>1$ sufficiently large and $g \in \calB$ with $g(0)=0$, one has
$$
|g(z)| \lesssim \int_{\D} \frac{|g'(w)|(1-|w|^2)^{\beta}}{\left|1-\overline{w}z \right|^{1+\beta}} dA(w), 
$$
where the implicit constant above is independent of $g$.
\end{lem} 

\begin{proof}
The above estimate is well-known (see, e.g., the proof of \cite[Lemma 3.2]{PauZhao2014}). Here, we would like to include proof for such a result, which is for the reader's convenience. To begin with, since $\beta>1$, we have $\calB=\calQ_\beta$, which in particular gives
$$
\int_{\D} \left|g'(z) \right|^2 \left(1-|z|^2 \right)^\beta dA(z)<+\infty, 
$$
and therefore $g' \in A_\beta^2$, where $A_\beta^2$ is the weighted Bergman space. Now using \cite[Proposition 4.23]{ZhuBook2007}, we have
$$
g'(\xi)=\int_{\D} \frac{(1-|w|^2)^\beta g'(w)}{\left(1-\xi \overline{w} \right)^{2+\beta}} dA(w). 
$$
Taking the integral along the line segment enclosed by $0$ and $z$ on both sides of the above equation and then using Fubini, we see that 
\begin{eqnarray*}
\left|g(z) \right|%
&=& \left| \int_0^z \left[\int_{\D} \frac{(1-|w|^2)^\beta g'(w)}{\left(1-\xi \overline{w} \right)^{2+\beta}} dA(w) \right] d \xi \right| \\
&=& \left| \int_\D (1-|w|^2)^\beta g'(w) \cdot \left[ \int_0^z \frac{1}{\left(1-\xi \overline{w} \right)^{2+\beta}} d\xi \right] dA(w) \right| \\
&\lesssim& \int_{\D} (1-|w|^2)^\beta |g'(w)| \cdot \left| \frac{1-(1-z \overline{w})^{1+\beta}}{(1+\beta)\overline{w}} \right| \cdot \frac{1}{\left|1-z\overline{w} \right|^{1+\beta}} dA(w).
\end{eqnarray*}
The desired estimate follows from the fact that 
$$
\sup_{z, w \in \D} \left| \frac{1-(1-z \overline{w})^{1+\beta}}{(1+\beta)\overline{w}} \right| \lesssim 1. 
$$
\end{proof}

Returning back to the term $M_I^{1, 1}(f)$, we first note that the function $\frac{f(z)-f(a_I)}{(1-\overline{a_I}z)^s}+f(a_I)$ satisfies the assumption of Lemma \ref{20240614lem01}. Therefore, we have
\begin{eqnarray} \label{20240614eq04} 
N(f, I, z)&:=& \left|\frac{f(z)-f(a_I)}{(1-\overline{a_I}z)^s}+f(a_I) \right|^2 \lesssim \left[ \int_{\D} \frac{\left| \frac{d}{dw} \left(\frac{f(w)-f(a_I)}{(1-\overline{a_I}w)^s} \right) \right|(1-|w|^2)^{\beta}}{\left|1-\overline{w}z \right|^{1+\beta}} dA(w)\right]^2 \nonumber \\
&\lesssim& N_1(f, I, z)+N_2(f, I, z),
\end{eqnarray}
where
$$
N_1(f, I, z):=\left[\int_{\D} \frac{|f'(w)|(1-|w|^2)^{\beta}}{\left|1-\overline{a_I}w \right|^s \left|1-\overline{w}z \right|^{1+\beta}}dA(w)\right]^2
$$
and
$$
N_2(f, I, z):=\left[\int_{\D} \frac{|f(w)-f(a_I)|(1-|w|^2)^{\beta}}{\left|1-\overline{a_I}w \right|^{s+1} \left|1-\overline{w}z \right|^{1+\beta}}dA(w)\right]^2.
$$
Note that here it is crucial that the implicit constants in the above estimates do \emph{not} depend on $I$. The estimates of the terms $N_1(f, I, z)$ and $N_2(f, I, z)$ are the crux of our analysis, which is achieved via the following important lemma. 

\begin{lem} [{\cite[Theorem 3.1]{zhangLishangGuo2018}}] \label{20240614lem02}
Suppose $\delta>-1$, $b, c \ge 0$, and $k \ge 0$ with $b+c-\del>2$, $b-\del<2$ and $c-\del<2$. Let
$$
J_{z, a}:=\int_\D \frac{(1-|w|^2)^\delta}{\left|1-\overline{z}w \right|^b \left|1-\overline{a}w \right|^c} \log^k \frac{e}{1-|w|^2} dA(w), \quad z, a \in \D.
$$
Then 
$$
J_{z, a} \simeq \frac{1}{\left|1-\overline{a}z \right|^{b+c-\del-2}} \log^k \frac{e}{\left|1- \overline{a}z \right|}. 
$$
\end{lem}

\subsection{Estimates of $N_1(f, I, z)$ and $N_2(f, I, z)$} 

We first estimate the term $N_1(f, I, z)$. Since $s>t$, denote
$$
\gamma:=\frac{s-t}{2}>0.
$$
By Cauchy-Schwarz, we have 
\begin{eqnarray} \label{20240615eq01}
&&N_1(f, I, z)= \left[ \int_{\D} \frac{|f'(w)|(1-|w|^2)^{s-\gamma}}{\left|1-\overline{a_I} w \right|^{s-2\gamma} \left|1- \overline{w} z \right|^s} \cdot \frac{(1-|w|^2)^{\beta-s+\gamma}}{\left|1-\overline{a_I} w\right|^{2 \gamma} \left|1- \overline{w} z \right|^{1+\beta-s}} dA(w)\right]^2 \nonumber \\
&&\quad \le \left( \int_{\D} \frac{|f'(w)|^2(1-|w|^2)^{2s-2\gamma} dA(w)}{\left|1-\overline{a_I} w \right|^{2(s-2\gamma)} \left|1- \overline{w} z \right|^{2s}} \right) \cdot \left( \int_\D \frac{(1-|w|^2)^{2(\beta-s+\gamma)}dA(w)}{\left|1-\overline{a_I} w\right|^{4 \gamma} \left|1- \overline{w} z \right|^{2(1+\beta-s)}} \right).
\end{eqnarray}
Now we estimate the second integral in \eqref{20240615eq01} by using Lemma \ref{20240614lem02} with
$$
\delta=2(\beta-s+\gamma), \ b=2(1+\beta-s), \ c=4\gamma, \quad \textnormal{and} \quad k=0. 
$$
Observe that one has
$$
b+c-\delta=2+2\gamma>2, \ b-\del=2-2\gamma<2, \quad \textrm{and} \quad c-\delta=2\gamma+2s-2\beta<2.
$$
Therefore, 
$$
\int_\D \frac{(1-|w|^2)^{2(\beta-s+\gamma)}dA(w)}{\left|1-\overline{a_I} w\right|^{4 \gamma} \left|1- \overline{w} z \right|^{2(1+\beta-s)}} \simeq \frac{1}{\left|1-\overline{a_I}z \right|^{2\gamma}}. 
$$
Plugging the above estimate back to \eqref{20240615eq01}, we see that
\begin{equation} \label{20240614eq10}
N_1(f, I, z) \lesssim \frac{1}{|I|^{2\gamma}} \cdot \int_{\D} \frac{|f'(w)|^2(1-|w|^2)^{2s-2\gamma} }{\left|1-\overline{a_I} w \right|^{2(s-2\gamma)} \left|1- \overline{w} z \right|^{2s}} dA(w).
\end{equation}

\medskip 

Next, we estimate $N_2(f, I, z)$. Let us first split $N_2(f, I, z)$ into two parts:
\begin{equation} \label{20240614eq11}
N_2(f, I, z) \lesssim N_{2, 1}(f, I, z)+N_{2, 2}(f, I, z),
\end{equation} 
where 
$$
N_{2, 1}(f, I, z):=\left[\int_{\D} \frac{|f(w)-f(0)|(1-|w|^2)^{\beta}}{\left|1-\overline{a_I}w \right|^{s+1} \left|1-\overline{w}z \right|^{1+\beta}}dA(w)\right]^2, 
$$
and
$$
N_{2, 2}(f, I, z):=\left[\int_{\D} \frac{|f(a_I)-f(0)|(1-|w|^2)^{\beta}}{\left|1-\overline{a_I}w \right|^{s+1} \left|1-\overline{w}z \right|^{1+\beta}}dA(w)\right]^2.
$$
We estimate $N_{2, 2}(f, I, z)$ first. Clearly, $|f(a_I)-f(0)| \lesssim \left\|f \right\|_{\calB} \log \frac{2}{1-|a_I|}$, and hence
$$
N_{2, 2}(f, I, z) \lesssim \left\|f \right\|^2_{\calB}  \cdot \log^2 \frac{2}{1-|a_I|} \left[\int_{\D} \frac{(1-|w|^2)^{\beta}}{\left|1-\overline{a_I}w \right|^{s+1} \left|1-\overline{w}z \right|^{1+\beta}}dA(w)\right]^2.
$$
Letting now 
$$
\delta=\beta, \ b=\beta+1, \ c=s+1 \quad \textnormal{and} \quad k=0
$$
in Lemma \ref{20240614lem02} and checking that
$$
b+c-\del=s+2>2, \ b-\delta=1<2, \quad \textrm{and} \quad c-\delta=s+1-\beta<2,
$$
we conclude that 
\begin{eqnarray} \label{20240614eq12}
N_{2, 2}(f, I, z)%
&\lesssim& \left\|f \right\|^2_{\calB}  \cdot \log^2 \frac{2}{1-|a_I|} \cdot \frac{1}{\left|1-\overline{a_I}z \right|^{2s}} \nonumber \\ 
&\lesssim&  \frac{\left\|f \right\|_{\calB}^2}{|I|^{2s}} \left(\log \frac{e}{|I|} \right)^2 \lesssim  \frac{\left\|f \right\|_{\calQ_t}^2}{|I|^{2s}} \left(\log \frac{e}{|I|} \right)^2, 
\end{eqnarray} 
where in the second estimate, we have used that $z \in S(I)$ again. 

\medskip 

Finally, we estimate $N_{2, 1}(f, I, z)$. Using again the fact that $|f(w)-f(0)| \lesssim \|f\|_{\calB} \log \frac{2}{1-|w|}$,
we have 
\begin{equation} \label{20240626eq20}
N_{2, 1}(f, I, z) \lesssim \left\|f \right\|_{\calB}^2 \cdot \left[\int_{\D} \frac{\log \frac{2}{1-|w|} \cdot (1-|w|^2)^{\beta}}{\left|1-\overline{a_I}w \right|^{s+1} \left|1-\overline{w} z \right|^{1+\beta}} dA(w)\right]^2.
\end{equation}
Now in Lemma \ref{20240614lem02}, we let 
$$
\delta=\beta, \ b=\beta+1, \ c=s+1, \quad \textrm{and} \quad k=1, 
$$
and this gives
\begin{eqnarray} \label{20240614eq13}
N_{2, 1}(f, I, z)%
&\lesssim& \left\|f \right\|_{\calB}^2 \cdot \left[\frac{1}{\left|1-\overline{a_I}z \right|^{s}} \log \frac{e}{\left|1-\overline{a_I}z \right|} \right]^2 \nonumber \\
&\lesssim&  \frac{\left\|f \right\|_{\calB}^2}{|I|^{2s}} \left(\log \frac{e}{|I|} \right)^2 \lesssim  \frac{\left\|f \right\|_{\calQ_t}^2}{|I|^{2s}} \left(\log \frac{e}{|I|} \right)^2.
\end{eqnarray}

\medskip 

Combining now \eqref{20240614eq04}, \eqref{20240614eq10}, \eqref{20240614eq11}, \eqref{20240614eq12}, and \eqref{20240614eq13}, we conclude that
$$
N(f, I, z) \lesssim \frac{\left\|f \right\|_{\calB}^2}{|I|^{2s}} \left(\log \frac{e}{|I|}\right)^2+ \frac{1}{|I|^{2\gamma}} \cdot \int_{\D} \frac{|f'(w)|^2(1-|w|^2)^{2s-2\gamma} }{\left|1-\overline{a_I} w \right|^{2(s-2\gamma)} \left|1- \overline{w} z \right|^{2s}} dA(w).
$$
Plugging the above estimate back to \eqref{20240614eq20}, and using the fact that $\log \frac{2}{|I|} \gtrsim 1$, we have
\begin{eqnarray*} 
&& M_I^{1,1}(f)=\frac{(1-|a_I|^2)^s}{\left(\log{\frac{2}{|I|}}\right)^2}\int_{S(I)}
    N(f, I, z) d\mu(z) \nonumber \\
&& \quad \lesssim \left\|f \right\|_{\calB}^2 \cdot \frac{\mu(S(I))}{|I|^s}+  (1-|a_I|^2)^{s-2\gamma}\int_{S(I)}
    \left[\int_{\D} \frac{|f'(w)|^2(1-|w|^2)^{2s-2\gamma} }{\left|1-\overline{a_I} w \right|^{2(s-2\gamma)} \left|1- \overline{w} z \right|^{2s}} dA(w) \right] d\mu(z).
\end{eqnarray*} 
By Fubini, we observe that
\begin{eqnarray*}
&& (1-|a_I|^2)^{s-2\gamma}\int_{S(I)}
    \left[\int_{\D} \frac{|f'(w)|^2(1-|w|^2)^{2s-2\gamma} }{\left|1-\overline{a_I} w \right|^{2(s-2\gamma)} \left|1- \overline{w} z \right|^{2s}} dA(w) \right] d\mu(z) \\
    &&\quad = \int_{\D} \left|f'(w) \right|^2 \frac{(1-|w|^2)^{s-2\gamma}(1-|a_I|^2)^{s-2\gamma}}{\left|1-\overline{a_I}w \right|^{2(s-2\gamma)}} \left( \int_{S(I)} \frac{(1-|w|^2)^s}{\left|1-\overline{w} z\right|^{2s}} d\mu(z) \right) dA(w) \\
    && \quad= \int_{\D} |f'(w)|^2 (1-|\varphi_{a_I}(w)|^2)^{s-2\gamma}  \left( \int_{S(I)} \frac{(1-|w|^2)^s}{\left|1-\overline{w} z\right|^{2s}} d\mu(z) \right) dA(w).
\end{eqnarray*}
Since $\mu$ is $s$-Carleson, it is well-known that
$$
\int_{S(I)} \frac{(1-|w|^2)^s}{\left|1-\overline{w} z\right|^{2s}} d\mu(z) \lesssim \left\|\mu \right\|_{\mathcal{CM}_s}. 
$$
(see, e.g., \cite[Theorem 45]{ZhaoZhu2008}). 
Therefore, 
$$
(1-|a_I|^2)^{s-2\gamma}\int_{S(I)}
    \left[\int_{\D} \frac{|f'(w)|^2(1-|w|^2)^{2s-2\gamma} }{\left|1-\overline{a_I} w \right|^{2(s-2\gamma)} \left|1- \overline{w} z \right|^{2s}} dA(w) \right] d\mu(z) \lesssim \left\|\mu \right\|_{\mathcal{CM}_s} \left\|f \right\|_{\calQ_t}^2, 
$$
and hence
\begin{equation} \label{20240614eq21}
M_I^{1, 1}(f) \lesssim \left\|f \right\|_{\calB}^2 \cdot \frac{\mu(S(I))}{|I|^s}+\left\|\mu \right\|_{\mathcal{CM}_s} \left\|f \right\|_{\calQ_t}^2 \lesssim \left\|\mu \right\|_{\mathcal{CM}_s} \left\|f \right\|_{\calQ_t}^2. 
\end{equation} 

Finally, by \eqref{20240614eq02}, \eqref{20240614eq03} , \eqref{20240614eq04}, \eqref{20240614eq04}, and \eqref{20240614eq21}, we have 
\begin{eqnarray*}
\frac{1}{|I|^s \left(\log{\frac{2}{|I|}}\right)^2}\int_{S(I)}|f(z)|^2d\mu(z)%
&=& M_I(f) \lesssim M_I^1(f)+M_I^2(f) \\
&\lesssim& M_I^{1, 1}(f)+M_I^{1, 2}(f)+M_I^2(f) \lesssim  \left\|\mu \right\|_{\mathcal{CM}_s} \left\|f \right\|_{\calQ_t}^2.
\end{eqnarray*}
Taking now the supremum over all $I \subseteq \partial \D$, the proof of the sufficient part of Theorem \ref{20240613thm02} is complete. $\hfill{\square}$

\medskip 

\begin{rem} \label{20240626rem01}
Here, we give an alternative way to estimate the term $N_2(f, I, z)$, which, although, is weaker than our original estimate \eqref{20240614eq13}, however, suffices our purpose. The idea here is close to the estimate of $N_1(f, I, z)$. First, by Cauchy Schwarz, we have 
\begin{eqnarray*}
N_2(f, I, z)%
&\lesssim&  \left( \int_{\D} \frac{|f(w)-f(a_I)|^2(1-|w|^2)^{2s-2\gamma} dA(w)}{\left|1-\overline{a_I} w \right|^{2(s+1-2\gamma)} \left|1- \overline{w} z \right|^{2s}} \right) \cdot \left( \int_\D \frac{(1-|w|^2)^{2(\beta-s+\gamma)}dA(w)}{\left|1-\overline{a_I} w\right|^{4 \gamma} \left|1- \overline{w} z \right|^{2(1+\beta-s)}} \right) \\
&\lesssim& \frac{1}{|I|^{2\gamma}} \cdot \left( \int_{\D} \frac{|f(w)-f(a_I)|^2(1-|w|^2)^{2s-2\gamma} dA(w)}{\left|1-\overline{a_I} w \right|^{2(s+1-2\gamma)} \left|1- \overline{w} z \right|^{2s}} \right), 
\end{eqnarray*}
where we recall that $\gamma=\frac{s-t}{2}$. Therefore, 
\begin{eqnarray*}
&& \frac{|I|^s}{\left(\log \frac{2}{|I|} \right)^2} \int_{S(I)} N_2(f, I, z) d\mu(z) \\
&& \quad \lesssim  \frac{|I|^{s-2\gamma}}{\left(\log \frac{2}{|I|} \right)^2} \int_{S(I)} \left( \int_{\D} \frac{|f(w)-f(a_I)|^2(1-|w|^2)^{2s-2\gamma} dA(w)}{\left|1-\overline{a_I} w \right|^{2(s+1-2\gamma)} \left|1- \overline{w} z \right|^{2s}} \right) d\mu(z) \\
&& \quad \lesssim \int_{\D} \frac{|f(w)-f(a_I)|^2}{\left|1- \overline{a_I} w\right|^2} \cdot \frac{(1-|w|^2)^{s-2\gamma}(1-|a_I|)^{s-2\gamma}}{\left|1-\overline{a_I} w \right|^{2(s-2\gamma)}} \cdot \left( \int_{S(I)} \frac{(1-|w|^2)^s}{\left|1-\overline{w}z \right|^{2s}} d\mu(z) \right) dA(w) \\
&& \quad \lesssim \left\|\mu \right\|_{\mathcal{CM}_s} \int_{\D} \frac{|f(w)-f(a_I)|^2}{\left|1- \overline{a_I} w\right|^2} \cdot (1-|\varphi_{a_I}(w)|^2)^{s-2\gamma} dA(w) \lesssim \left\|\mu \right\|_{\mathcal{CM}_s} \left\|f \right\|_{\calQ_t}^2, 
\end{eqnarray*}
where the last estimate was contained in the proof of \cite[Theorem 1.1]{Xiaojie2008} (or more generally, \cite[Proposition 2.8]{PauZhao2014}). 
\end{rem}

\subsection{Compactness}

As some byproducts of Theorems \ref{20240613thm02} and \ref{20240624thm01}, we get their corresponding compactness counterparts. To begin with, we recall the following lemma. 

\begin{lem} [{\cite[Lemma 2.2]{LiLiuYuan2017} and \cite[Lemma 4]{LiuLou2015}}]\label{20240621lem01}
For $0<r<1$, let $\one_{\left\{z:|z|<r\right\}}$ be the characterization function of the set $\{z:|z|<r\}$. If $\mu$ is an $s$-Carleson measure on $\D$, then $\mu$ is a vanishing $s$-Carleson measure if and only if $\|\mu-\mu_{r}\|_{\mathcal{CM}_s}\to 0$ as $r\to 1^{-}$, where $\mu_{r}:=\one_{\left\{z:|z|<r\right\}}\mu$. 
\end{lem}

Here, we say $\mu$ is a \emph{vanishing $s$-Carleson measure} if 
$$
\lim\limits_{|I| \to 0} \frac{\mu(S(I))}{|I|^s}=0.
$$
We have the following result.

\begin{thm} \label{20240621thm01}
For any $0<t<s \le 1$ and $\mu$ being a postive Borel measure on $\D$, $id: \calQ_t \mapsto \calT_{s, 2}^2(\mu)$ is compact if and only if $\mu$ is a vanishing $s$-Carleson measure. More generally, we have $id: F(p, p-2, s) \mapsto \calT_{s, p}^p$ is compact if and only if $\mu$ is a vanishing $s$-Carleson measure. 
\end{thm}
 \begin{proof} 
 Again, we only focus on the case when $p=2$. We first show the necessity. Let $\{I_{j}\}$ be a sequence of subarcs of $\partial\D$ such that $|I_j|\to 0$ as $j \to \infty$. Let further
$$f_{j}(z):=\left(\log\frac{1}{1-|a_{I_j}|^2}\right)^{-1}\left(\log\frac{1}{1-\overline{a_{I_j}}z}\right)^2,   \qquad   z\in \D.$$
Then it is clear that $\|f_j\|_{\calQ_t}\lesssim 1$ and $f_j\to 0$ uniformly on compact subsets of $\D$. Since $id: \calQ_t \mapsto \calT_{s, 2}^2(\mu)$ is compact, it is well-known that $\left\|f_j \right\|_{\calT_{s, 2}^2} \to 0$ as $j \to 0$. Therefore, 
\begin{eqnarray*}
\frac{\mu(S(I_j))}{|I_j|^s}\lesssim \frac{\int_{S(I_j)}|f_j(z)|^2d\mu(z)}{|I_j|^s\left( 
\log\frac{2}{|I_j|} \right)^2}\lesssim \|f_j\|_{\calT_{s,2}^2}. 
\end{eqnarray*}
The desired claim follows by letting $j \to 0$ on both sides of the above estimate.

\medskip 

Next, we prove the sufficient part. Let $\{f_n\}$ be a bounded sequence in $\calQ_t$ with $f_{n}(z)\to 0$ uniformly on compact subsets of $\D$. To prove $id: \calQ_t \mapsto \calT_{s, 2}^2(\mu)$ is compact, it suffices to show that $\lim\limits_{n \to \infty}\|f_n\|_{\calT_{s, 2}^2(\mu)}= 0$. 

Note that for any $I \subseteq \partial D$ and $r \in (0, 1)$, using Theorem \ref{20240613thm02}, one has 
\begin{eqnarray*} \label{20240625eq30}
&&\frac{1}{|I|^s\left( 
\log\frac{2}{|I|} \right)^2}\int_{S(I)}|f_n(z)|^2d\mu(z)\lesssim \frac{1}{|I|^s\left( 
\log\frac{2}{|I|} \right)^2}\int_{S(I)}|f_n(z)|^2 d\mu_r(z) \nonumber \\
&& \qquad \qquad \qquad +\frac{1}{|I|^s \left( \log\frac{2}{|I|} \right)^2}\int_{S(I)}|f_n(z)|^2d(\mu-\mu_r)(z)\\
&& \qquad \lesssim \frac{1}{|I|^s} \int_{|z| \le r} |f_n(z)|^2 d\mu(z)+ \left\|\mu-\mu_r \right\|_{\mathcal{CM}_s} \left\|f_n \right\|_{\calQ_t}^2. 
\end{eqnarray*}
Take any $\epsilon>0$. Since $\mu$ is a vanishing $s$-Carleson measure, by Lemma \ref{20240621lem01}, we can find an $r \in (0, 1)$, such that for any $n \ge 1$, one has $\left\|\mu-\mu_r \right\|_{\mathcal{CM}_s} \left\|f_n \right\|_{\calQ_t}^2<\epsilon$.  
Fix $r$ and then take an $N>0$, such that for any $n>N$, $|f_n(z)|<\sqrt{\frac{\epsilon}{ \left\|\mu \right\|_{\mathcal{CM}_s}}}, \ z \in \{z \in \D: |z|\le r\}$. Therefore, by \eqref{20240625eq30}, we see that when $n>N$, 
$$
\frac{1}{|I|^s\left( 
\log\frac{2}{|I|} \right)^2}\int_{S(I)}|f_n(z)|^2d\mu(z) \lesssim \left\| \mu \right\|_{\mathcal{CM}_s} \cdot \frac{\epsilon}{\left\|\mu \right\|_{\mathcal{CM}_s}}+\epsilon \lesssim \epsilon, 
$$
which implies the desired result. The proof is complete. 
\end{proof}

\bigskip 

\section{Proof of Theorem \ref{20240624thm10}}
The goal of the second part of this paper is to study the near-endpoints Carleson embedding for the $F(p, p\alpha-2, s)$ spaces for $\alpha>1$. The idea of proving Theorem \ref{20240624thm10} is similar to the proof of Theorems \ref{20240613thm02} and \ref{20240624thm01}. 

\subsection{Initial reductions}
We prove the necessity first. Recall that for $\alpha>1$, $F(p, p\alpha-2, t) \subseteq \calB^\alpha$. By \cite[Proposition 7]{Zhu1993}, we have 
\begin{equation} \label{20240625eq02}
\left\|f \right\|_{\calB_\alpha} \simeq \sup\limits_{z \in \D} |f(z)|(1-|z|^2)^{\alpha-1}.
\end{equation} 
Therefore, for each $I \subseteq \partial \D$, we consider the test function 
$$
f_{a_I}(z)=\frac{1}{(1-\overline{a_I}z)^{\alpha-1}}.
$$
By \cite[Lemma 2.6]{PauZhao2014}, we have $\sup\limits_{a_I \in \D} \left\|f_{a_I} \right\|_{F(p, p\alpha-2, t)} \lesssim  1$. Therefore, 
$$
\frac{\mu(S(I))}{|I|^{s+p(\alpha-1)}} \simeq \frac{1}{|I|^s} \int_{S(I)} \left|f_{a_I}(z) \right|^p d\mu(z) \lesssim \left\|f_{a_I} \right\|_{F(p, p\alpha-2, t)} \lesssim 1. 
$$
Taking the supremum over all $I \subseteq \partial \D$, we conclude that $\mu$ is an $s+p(\alpha-1)$-Carleson measure. 

\medskip 

Next, we prove the sufficient part. For any $I \subseteq \partial \D$, we first write
\begin{equation} \label{20240625eq01}
\frac{1}{|I|^s} \int_{S(I)} |f(z)|^p d\mu(z) \lesssim \calM_I^1(f)+\calM_I^2(f), 
\end{equation} 
where 
$$
\calM_I^1(f):=\frac{1}{|I|^s} \int_{S(I)} |f(z)-f(a_I)|^pd\mu(z) \quad \textrm{and} \quad 
\calM_I^2(f):=\frac{1}{|I|^s} \int_{S(I)} |f(a_I)|^p d\mu(z). 
$$
The estimate of $\calM_I^2(f)$ is routine. Indeed, by \eqref{20240625eq02}, we have 
\begin{eqnarray} \label{20240625eq03}
\calM_I^2(f)%
&=& \frac{\mu(S(I))}{|I|^s} \cdot \left|f(a_I) \right|^p \lesssim \frac{\mu(S(I))}{|I|^s} \cdot \left\|f \right\|^p_{\calB_\alpha} \cdot (1-|a_I|^2)^{p(1-\alpha)} \nonumber \\
&\lesssim& \left\|f \right\|^p_{F(p, p\alpha-2, t)} \left\|\mu \right\|_{\mathcal{CM}_{s+p(\alpha-1)}}.
\end{eqnarray}
Now to estimate $\calM_I^1(f)$, we write
\begin{eqnarray*}
\calM_I^1(f)%
&\simeq&  |I|^{s+2p(\alpha-1)} \int_{S(I)} \left| \frac{f(z)-f(a_I)}{(1-\overline{a_I}z)^{\frac{2s}{p}+2(\alpha-1)}} \right|^p d\mu(z) \\
&\lesssim& \calM_I^{1, 1}(f)+\calM_I^{2, 2}(f),
\end{eqnarray*}
where
$$
\calM_I^{1, 1}(f):=|I|^{s+2p(\alpha-1)} \int_{S(I)} \left| \frac{f(z)-f(a_I)}{(1-\overline{a_I}z)^{\frac{2s}{p}+2(\alpha-1)}}+f(a_I)\right|^p d\mu(z)
$$
and
$$
\calM_I^{1, 2}(f):=|I|^{s+2p(\alpha-1)} \int_{S(I)} \left| f(a_I)\right|^p d\mu(z).
$$
Again, we can estimate the term $\calM_I^{1, 2}(f)$ by using \eqref{20240625eq02}. Indeed, 
\begin{eqnarray} \label{20240626eq01}
\calM_I^{1, 2}(f)%
&\lesssim&  |I|^{s+2p(\alpha-1)}\cdot \frac{ \mu(\D)\left\|f \right\|^p_{\calB_\alpha}}{ (1-|a_I|^2)^{p(\alpha-1)}} \nonumber \\
&\lesssim& \left\|\mu \right\|_{\mathcal{CM}_{s+p(\alpha-1)}} \left\|f \right\|_{\calB_\alpha}^p \lesssim \left\|\mu \right\|_{\mathcal{CM}_{s+p(\alpha-1)}} \left\|f \right\|_{F(p, p\alpha-2, t)}^p.
\end{eqnarray}
We are now left to estimate the term $M_I^{1, 1}(f)$. Observe now that for each $I \subseteq \partial \D$, the holomorphic function $\frac{f(z)-f(a_I)}{(1-\overline{a_I}z)^{\frac{2s}{p}+2(\alpha-1)}}+f(a_I) \in \calB_\alpha$. This means that we can pick a $\beta>1$ sufficiently large, such that 
$$
\int_{\D} \left| \left(\frac{f(z)-f(a_I)}{(1-\overline{a_I}z)^{\frac{2s}{p}+2(\alpha-1)}}+f(a_I) \right)' \right|(1-|z|^2)^{\beta}<+\infty.
$$
This together with the fact that $\left[ \frac{f(z)-f(a_I)}{(1-\overline{a_I}z)^{\frac{2s}{p}+2(\alpha-1)}}+f(a_I) \right] \bigg |_{z=0}=0$ yields Lemma \ref{20240614lem01} still applies in this case. Hence, we have 
\begin{eqnarray} \label{20240426eq02}
\calN(f, I, z)%
&:=& \left|  \frac{f(z)-f(a_I)}{(1-\overline{a_I}z)^{\frac{2s}{p}+2(\alpha-1)}}+f(a_I) \right|^p \nonumber \\
&\lesssim& \left[\int_\D \frac{\left| \frac{d}{dw} \left( \frac{f(w)-f(a_I)}{(1-\overline{a_I}w)^{\frac{2s}{p}+2(\alpha-1)}} \right) \right|(1-|w|^2)^{\beta}}{\left|1-\overline{w}z \right|^{1+\beta}} dA(w) \right]^p \nonumber \\
&\lesssim& \calN_1(f, I, z)+\calN_2(f, I, z),
\end{eqnarray}
where 
$$
\calN_1(f, I, z):= \left[ \int_\D \frac{\left|f'(w) \right|(1-|w|^2)^{\beta}}{\left|1-\overline{a_I}w \right|^{\frac{2s}{p}+2(\alpha-1)}\left|1-\overline{w}z \right|^{1+\beta}} dA(w) \right]^p
$$
and
$$
\calN_2(f, I, z):=\left[ \int_\D \frac{\left|f(w)-f(a_I) \right|(1-|w|^2)^{\beta}}{\left|1-\overline{a_I}w \right|^{\frac{2s}{p}+2(\alpha-1)+1}\left|1-\overline{w}z \right|^{1+\beta}} dA(w) \right]^p
$$

\subsection{Estimates of $\calN_1(f, I, z)$ and $\calN_2(f, I, z)$}

Since $s>t$, we denote
$$
\gamma:=\frac{s-t}{p}>0.
$$
We first estimate the term of $\calN_1(f, I, z)$. An application of H\"older's inequality yields
\begin{eqnarray} \label{20240626eq03}
&& \calN_1(f, I, z)=\bigg[ \int_{\D}  \frac{|f'(w)|(1-|w|^2)^{\alpha-\frac{2}{p}} (1-|w|^2)^{\frac{2s}{p}-\gamma}}{\left|1-\overline{a_I} w \right|^{\frac{2s}{p}+2(\alpha-\gamma-1)} \left|1-\overline{w}z \right|^{\alpha+\frac{2s}{p}-1}} \nonumber \\
&& \qquad \qquad \qquad \qquad \qquad \qquad   \cdot \frac{(1-|w|^2)^{\beta-\alpha+\frac{2}{p}-\frac{2s}{p}+\gamma}}{\left|1-\overline{a_I} w \right|^{2\gamma} \left|1-\overline{w}z \right|^{\beta-\alpha+\frac{2}{p}-\frac{2s}{p}+\frac{2}{p'}}}dA(w)\bigg]^{p} \nonumber \\
&& \quad \lesssim  \left( \int_{\D} \frac{|f'(w)|^p (1-|w|^2)^{p\alpha-2} (1-|w|^2)^{2s-p\gamma}}{\left|1-\overline{a_I} w \right|^{2s+2p(\alpha-\gamma-1)} \left|1-\overline{w}z \right|^{\alpha p+2s-p}}  dA(w) \right) \nonumber \\
&&  \qquad \qquad \qquad \qquad \qquad \qquad   \cdot \left( \int_{\D} \frac{(1-|w|^2)^{p'\left(\beta-\alpha+\frac{2}{p}-\frac{2s}{p}+\gamma\right)}}{\left|1-\overline{a_I} w \right|^{2p' \gamma} \left|1-\overline{w}z \right|^{p'(\beta-\alpha+\frac{2}{p}-\frac{2s}{p})+2}} dA(w) \right)^{\frac{p}{p'}}.
\end{eqnarray}
Here $p'=\frac{p}{p-1}$ is the conjugate of $p$. Putting
$$
\delta=p'\left(\beta-\alpha+\frac{2}{p}-\frac{2s}{p}+\gamma \right), \ b=p'\left(\beta-\alpha+\frac{2}{p}-\frac{2s}{p} \right)+2, \ c=2p'\gamma, \ \textrm{and} \ k=0,
$$
and verifying that\footnote{Recall that here we can pick any $\beta>1$ sufficiently large.}
$$
b+c-\del=2+p'\gamma>2, \ b-\delta=2-p'\gamma<2, \ \textrm{and} \ c-\del=p'\left(\gamma-\beta+\alpha-\frac{2}{p}+\frac{2s}{p} \right)<2, 
$$
we see that Lemma \ref{20240614lem02} yields
\begin{equation} \label{20240626eq10}
\int_{\D} \frac{(1-|w|^2)^{p'\left(\beta-\alpha+\frac{2}{p}-\frac{2s}{p}+\gamma\right)}}{\left|1-\overline{a_I} w \right|^{2p' \gamma} \left|1-\overline{w}z \right|^{p'(\beta-\alpha+\frac{2}{p}-\frac{2s}{p})+2}} dA(w) \simeq \frac{1}{\left|1-\overline{a_I} z \right|^{p'\gamma}}. 
\end{equation} 
Plugging the above estimate back to \eqref{20240626eq03}, we conclude that 
\begin{equation} \label{20240626eq04}
\calN_1(f, I, z) \lesssim \frac{1}{|1-\overline{a_I} z|^{p\gamma}} \cdot  \left( \int_{\D} \frac{|f'(w)|^p (1-|w|^2)^{p\alpha-2} (1-|w|^2)^{2s-p\gamma}}{\left|1-\overline{a_I} w \right|^{2s+2p(\alpha-\gamma-1)} \left|1-\overline{w}z \right|^{\alpha p+2s-p}}  dA(w) \right). 
\end{equation} 

\medskip 

Next, we must estimate the term $\calN_2(f, I, z)$. Here, we adapt the idea in Remark \ref{20240626rem01}. More precisely, by \eqref{20240626eq10}, we have 
\begin{eqnarray} \label{20240626eq11}
&& \calN_2(f, I, z)=\bigg[ \int_{\D}  \frac{|f(w)-f(a_I)|(1-|w|^2)^{\alpha-\frac{2}{p}} (1-|w|^2)^{\frac{2s}{p}-\gamma}}{\left|1-\overline{a_I} w \right|^{\frac{2s}{p}+2(\alpha-\gamma-1)+1} \left|1-\overline{w}z \right|^{\alpha+\frac{2s}{p}-1}} \nonumber \\
&& \qquad \qquad \qquad \qquad \qquad \qquad   \cdot \frac{(1-|w|^2)^{\beta-\alpha+\frac{2}{p}-\frac{2s}{p}+\gamma}}{\left|1-\overline{a_I} w \right|^{2\gamma} \left|1-\overline{w}z \right|^{\beta-\alpha+\frac{2}{p}-\frac{2s}{p}+\frac{2}{p'}}}dA(w)\bigg]^{p} \nonumber \\
&& \quad \lesssim  \left( \int_{\D} \frac{|f(w)-f(a_I)|^p (1-|w|^2)^{p\alpha-2} (1-|w|^2)^{2s-p\gamma}}{\left|1-\overline{a_I} w \right|^{2s+2p(\alpha-\gamma-1)+p} \left|1-\overline{w}z \right|^{\alpha p+2s-p}}  dA(w) \right) \nonumber \\
&&  \qquad \qquad \qquad \qquad \qquad \qquad   \cdot \left( \int_{\D} \frac{(1-|w|^2)^{p'\left(\beta-\alpha+\frac{2}{p}-\frac{2s}{p}+\gamma\right)}}{\left|1-\overline{a_I} w \right|^{2p' \gamma} \left|1-\overline{w}z \right|^{p'(\beta-\alpha+\frac{2}{p}-\frac{2s}{p})+2}} dA(w) \right)^{\frac{p}{p'}} \nonumber \\
&& \quad \lesssim \frac{1}{|1-\overline{a_I}z|^{p\gamma}} \cdot \left( \int_{\D} \frac{|f(w)-f(a_I)|^p (1-|w|^2)^{p\alpha-2} (1-|w|^2)^{2s-p\gamma}}{\left|1-\overline{a_I} w \right|^{2s+2p(\alpha-\gamma-1)+p} \left|1-\overline{w}z \right|^{\alpha p+2s-p}}  dA(w) \right). 
\end{eqnarray}

\begin{rem}
In the case when $\alpha>1$, it is \emph{not} efficient to estimate the term $|f(w)-f(a_I)|$ in $\calN_2(f, I, z)$ directly via
$$
|f(w)| \lesssim \left\|f \right\|_{\calB_\alpha}(1-|w|^2)^{1-\alpha}
$$
as in the previous case when we deal with the near-endpoints $\calQ_s$ Carleson embedding (see, \eqref{20240626eq20}). This is because the integral 
$$
\int_\D \frac{(1-|w|^2)^{\beta+1-\alpha}}{\left|1-\overline{a_I}w \right|^{\frac{2s}{p}+2(\alpha-1)+1}\left|1-\overline{w}z \right|^{1+\beta}} dA(w) 
$$
can no longer be controlled by either a certain power of $\frac{1}{1-|a_I|^2}$ or $\log \frac{1}{1-|a_I|^2}$ when $\alpha \ge 2$, however, will be of the magnitude either $\log \frac{1}{1-|z|^2}$ if $\alpha=2$ or $\frac{1}{(1-|z|^2)^{\alpha-2}}$ if $\alpha>2$ (see, \cite[Theorem 3.1]{zhangLishangGuo2018}), which blows up near $\partial \D$.
\end{rem}

Plugging now both estimates \eqref{20240626eq04} and \eqref{20240626eq11} back to \eqref{20240426eq02}, we see that 
\begin{eqnarray*}
&& \calN(f, I, z) \lesssim \frac{1}{|1-\overline{a_I} z|^{p\gamma}} \cdot  \left( \int_{\D} \frac{|f'(w)|^p (1-|w|^2)^{p\alpha-2} (1-|w|^2)^{2s-p\gamma}}{\left|1-\overline{a_I} w \right|^{2s+2p(\alpha-\gamma-1)} \left|1-\overline{w}z \right|^{\alpha p+2s-p}}  dA(w) \right) \\
&& \qquad \qquad  \qquad +\frac{1}{|1-\overline{a_I}z|^{p\gamma}} \cdot \left( \int_{\D} \frac{|f(w)-f(a_I)|^p (1-|w|^2)^{p\alpha-2} (1-|w|^2)^{2s-p\gamma}}{\left|1-\overline{a_I} w \right|^{2s+2p(\alpha-\gamma-1)+p} \left|1-\overline{w}z \right|^{\alpha p+2s-p}}  dA(w) \right).
\end{eqnarray*}
Therefore, using again the fact that $z \in S(I)$, 
\begin{eqnarray} \label{20240626eq50}
&& \calM_I^{1, 1}(f)\lesssim \left|I \right|^{s+2p(\alpha-1)-p\gamma} \int_{S(I)}  \int_{\D} \frac{|f'(w)|^p (1-|w|^2)^{p\alpha-2} (1-|w|^2)^{2s-p\gamma}}{\left|1-\overline{a_I} w \right|^{2s+2p(\alpha-\gamma-1)} \left|1-\overline{w}z \right|^{\alpha p+2s-p}}  dA(w) d\mu(z) \nonumber \\
&& \quad + \left|I \right|^{s+2p(\alpha-1)-p\gamma} \int_{S(I)} \int_{\D} \frac{|f(w)-f(a_I)|^p (1-|w|^2)^{p\alpha-2} (1-|w|^2)^{2s-p\gamma}}{\left|1-\overline{a_I} w \right|^{2s+2p(\alpha-\gamma-1)+p} \left|1-\overline{w}z \right|^{\alpha p+2s-p}}  dA(w)  d\mu(z).
\end{eqnarray}
For the first term in \eqref{20240626eq50}, by Fubini, we see that
\begin{eqnarray} \label{20240626eq51}
&& \left|I \right|^{s+2p(\alpha-1)-p\gamma} \int_{S(I)}  \int_{\D} \frac{|f'(w)|^p (1-|w|^2)^{p\alpha-2} (1-|w|^2)^{2s-p\gamma}}{\left|1-\overline{a_I} w \right|^{2s+2p(\alpha-\gamma-1)} \left|1-\overline{w}z \right|^{\alpha p+2s-p}}  dA(w) d\mu(z) \nonumber \\
&& \quad =\int_{\D} |f'(w)|^p (1-|w|^2)^{p\alpha-2} \cdot \frac{(1-|a_I|^2)^{s-p\gamma}(1-|w|^2)^{s-p\gamma}}{\left|1-\overline{a_I}w \right|^{2(s-p\gamma)}} \nonumber \\
&& \qquad \qquad \qquad \qquad\cdot \left(\int_{S(I)} \frac{(1-|w|^2)^s}{\left|1-\overline{w} z\right|^{\alpha p+2s-p}} d\mu(z) \right) \cdot \left(\frac{1-|a_I|^2}{\left|1-\overline{a_I}w \right|} \right)^{2p(\alpha-1)} dA(w) \nonumber \\
&& \quad \lesssim \left\|f\right\|_{F(p, p\alpha-2, t)}^p \left\| \mu \right\|_{\mathcal{CM}_{s+p(\alpha-1)}}. 
\end{eqnarray}
Similarly, for the second term in \eqref{20240626eq50}, we get
\begin{eqnarray} \label{20240626eq52}
&& \left|I \right|^{s+2p(\alpha-1)-p\gamma} \int_{S(I)} \int_{\D} \frac{|f(w)-f(a_I)|^p (1-|w|^2)^{p\alpha-2} (1-|w|^2)^{2s-p\gamma}}{\left|1-\overline{a_I} w \right|^{2s+2p(\alpha-\gamma-1)+p} \left|1-\overline{w}z \right|^{\alpha p+2s-p}}  dA(w)  d\mu(z) \nonumber \\
&& \quad =\int_{\D} |f(w)-f(a_I)|^p (1-|w|^2)^{p\alpha-2} \cdot \frac{(1-|a_I|^2)^{s-p\gamma}(1-|w|^2)^{s-p\gamma}}{\left|1-\overline{a_I}w \right|^{2(s-p\gamma)+p}} \nonumber \\
&& \qquad \qquad \qquad \qquad\cdot \left(\int_{S(I)} \frac{(1-|w|^2)^s}{\left|1-\overline{w} z\right|^{\alpha p+2s-p}} d\mu(z) \right) \cdot \left(\frac{1-|a_I|^2}{\left|1-\overline{a_I}w \right|} \right)^{2p(\alpha-1)} dA(w) \nonumber \\
&& \quad \lesssim \left\|\mu \right\|_{\mathcal{CM}_{s+p(\alpha-1)}} \cdot (1-|a_I|^2)^{(s-p\gamma)+2p(\alpha-1)} \nonumber \\
&&  \qquad \qquad \qquad \qquad\cdot \int_{\D} \frac{|f(w)-f(a_I)|^p(1-|z|^2)^{(s-p\gamma)+p\alpha-2}}{\left|1-\overline{a_I} w \right|^{p+(s-p\gamma)+[s-p\gamma+2p(\alpha-1)]}} dA(w) \nonumber \\
&& \quad \lesssim \left\|f\right\|_{F(p, p\alpha-2, t)}^p \left\| \mu \right\|_{\mathcal{CM}_{s+p(\alpha-1)}}. 
\end{eqnarray}
Therefore, by \eqref{20240626eq50}, \eqref{20240626eq51}, and \eqref{20240626eq52}, we see that  
$$
\calM_I^{1, 1}(f) \lesssim \left\|f\right\|_{F(p, p\alpha-2, t)}^p \left\| \mu \right\|_{\mathcal{CM}_{s+p(\alpha-1)}}. 
$$
Finally, combining the above estimate with \eqref{20240625eq01}, \eqref{20240625eq03}, and \eqref{20240626eq01}, we conclude that 
$$
\frac{1}{|I|^s} \int_{S(I)} |f(z)|^p d\mu(z) \lesssim \left\|f\right\|_{F(p, p\alpha-2, t)}^p \left\| \mu \right\|_{\mathcal{CM}_{s+p(\alpha-1)}}, 
$$
which clearly implies the desired sufficient part. The proof is complete. $\hfill{\square}$

\medskip 

To this end, we would like to only state the compactness result in this case, whose proof can be modified accordingly from the argument in Theorem \ref{20240621thm01}, and hence we would like to leave the details to the interested reader. 

\begin{thm} \label{20240626thm01}
Let $\alpha>1, 0<t<s \le 1$ and $p>1$. Then $id: F(p, p\alpha-2, t) \to \frakT_{s, p}^{\infty}(\mu)$ is compact if and only if $\mu$ is a vanishing $\left[s+p(\alpha-1)\right]$-Carleson measure. 
\end{thm} 

\medskip 

\section{Further remarks} \label{20240626sec01}

Finally, we compare the near-endpoints Carleson embeddings with both Xiao-Pau-Zhao's logarithmic Carleson embeddings and the Carleson embeddings induced by the continuously embedding properties of weighted Dirichlet spaces. As a consequence, we see that the near-endpoints Carleson embedding is an efficient tool to study the $F(p, p\alpha-2, s)$ Carleson embedding problems when the continuous embedding $\calD_\beta^p \hookrightarrow L^q(\mu)$ is not available. This part serves as a detailed explanation for Figure \ref{Fig1} in Section \ref{20240627sec01}. 

First of all, note that by Remark \ref{20240626rem10}, it suffices to consider the cases when $p>1$ and $\alpha \ge 1$. Moreover, we will only consider the case when\footnote{The case when $s \ge 1$ is less interesting. Indeed, in such a situation, the $F(p, p\alpha-2, s)$ Carleson embedding problems can be resolved up to endpoints by just using the continuous embedding properties of the weighted Dirichlet spaces (and hence the interval $\left[1, 1+\frac{1-s}{p} \right]$ in Figure \ref{Fig1} will be squeezed into a single point $\alpha=1$ in this case).} $0<s<1$. There are three different situations. 

\vspace{-0.1in}

\subsection{Case I: when $\alpha=1$.}
\vspace{-0.1in}

\begin{center}
\begin{tabular}{|p{4cm}|p{5.5cm}|p{5.5cm}|}
\hline
\multicolumn{3}{|c|}{}\\
\multicolumn{3}{|c|}{Near-Endpoints Carleson embedding V.S. Logarithmic Carleson embedding} \\ 
\multicolumn{3}{|c|}{}\\
\hline
&Behavior of $\mu$ & Continuous embedding\\
\hline
& &\\
Near-Endpoints Carleson embedding & $s$-Carleson measure & $F(p, p-2, t) \hookrightarrow \calT_{s, p}^p(\mu)$ for any $t<s$   \\
$p=2$, Theorem \ref{20240613thm02} & $\left\|\mu \right\|_{\mathcal{CM}_s}:=\sup\limits_{I \subset \partial \D} \frac{\mu(S(I))}{|I|^s}$ & $\left\|f \right\|^p_{\calT_{s, p}^p}=\sup\limits_{I \subset \partial \D} \frac{\int_{S(I)}|f(z)|^pd\mu(z)}{|I|^s \left(\log{\frac{2}{|I|}}\right)^p}$ \\
$p>1$, Theorem \ref{20240624thm01} & &  \\
\hline
& & \\
Xiao-Pau-Zhao's Logarithmic Carleson embedding & logarithmic $s$-Carleson measure & $F(p, p-2, s) \hookrightarrow \frakT^\infty_{s, p}(\mu)$ \\
$p=2$, \cite[Theorem 1.1]{Xiaojie2008} & $\left\|\mu \right\|_{\mathcal{LCM}_s}=\sup\limits_{I \subseteq \partial \D}  \frac{\left(\log \frac{2}{|I|} \right)^p\mu(S(I))}{|I|^s}$ & $\left\|f \right\|^p_{\frakT_{s, p}^\infty}=\sup\limits_{I \subseteq \partial \D} \frac{\int_{S(I)} |f(z)|^p d\mu(z)}{|I|^s}$  \\
$p>1$, \cite[Theorem 3.1]{PauZhao2014} & & \\
\hline
\end{tabular}
\\
\vspace{0.08in}
Table I: Near-endpoints Carleson embeddings at $\alpha=1$.
\end{center} 

\medskip 
 
In this case, although Xiao-Pau-Zhao's logarithmic Carleson embedding contains the endpoint $s$, these two Carleson embeddings do \emph{not} imply each other.  

\vspace{-0.1in}

\subsection{Case II: when $1<\alpha<1+\frac{1-s}{p}$ for $p>2$, and when $1<\alpha \le 1+\frac{1-s}{p}$ for $1<p \le2$.}

In this case, the continuous embedding $\calD_{s+\alpha p-2}^p \hookrightarrow L^q(\mu)$ is \emph{not} available. To our best knowledge, no prior results on the $F(p, p\alpha-2, s)$ Carleson embedding are known in this regime. One possible solution in this case is to adapt the idea from Xiao-Pau-Zhao's logarithmic Carleson embedding. Here is how one can interpret such an extension. Recall that in the case when $\alpha=1$, the conjecture is given by 

\begin{conj} \label{20240627conj01}
For any $p>1$ and $0<s<1$, $id: F(p, p-2, s) \mapsto \calT_{s, p}^p(\mu)$ is bounded if and only if $\mu$ is an $s$-Carleson measure. 
\end{conj} 

Theorem \ref{20240624thm01} asserts that the above conjecture holds for a near-endpoints Carleson embedding $F(p, p-2, t) \hookrightarrow \calT_{s, p}^p$ as long as $t<s$. In Xiao-Pau-Zhao's solution, they showed that the endpoint $t=s$ can be achieved, however, one has to adjust the definition by moving the logarithmic term $\left(\log \frac{2}{|I|} \right)^p$ into the definition of the Carleson measures. 

Therefore, for $\alpha>1$, following \cite[Theorem 4.4]{PauZhao2014}, the analog of Conjecture \ref{20240627conj01} can be stated as follows.

\begin{conj} \label{20240627conj02}
For any $p>1, \alpha>1$ and $0<s<1$, $id: F(p, p\alpha-2, s) \mapsto \frakT_{s, p}^\infty(\mu)$ is bounded if and only if $\mu$ is an $[s+p(\alpha-1)]$-Carleson measure. 
\end{conj}

Or equivalently, 

\begin{namedthm*}{Conjecture \ref{20240627conj02}$^\prime$} 
For any $p>1, \alpha>1$ and $0<s<1$ satisfying $s>p(\alpha-1)$, $id: F(p, p\alpha-2, s-p(\alpha-1)) \mapsto \frakT_{s-p(\alpha-1), p}^\infty(\mu)$ is bounded if and only if $\mu$ is an $s$-Carleson measure.     
\end{namedthm*}

Note that in this case, the role of $\left(\log \frac{2}{|I|} \right)^p$ is replaced by $\left(\frac{1}{|I|^{\alpha-1}} \right)^p$, and therefore one can write
$$
\left\|f \right\|^p_{\frakT_{s-p(\alpha-1), p}^{\infty}}=\sup_{I \subseteq \partial \D} \frac{1}{|I|^s \cdot \left(\frac{1}{|I|^{\alpha-1}} \right)^p} \int_{S(I)} |f(z)|^p d\mu(z). 
$$
Let us now move the term $\left(\frac{1}{|I|^{\alpha-1}} \right)^p$ into the definition of the Carleson measures, and this suggests the Xiao-Pau-Zhao's analog for $\alpha>1$ can be interpreted as 

\begin{namedthm}{Statement}[Xiao-Pau-Zhao's analog for $\alpha>1$] \label{20240627stat01}
For any $p>1, \alpha>1$ and $0<s<1$ satisfying $s>p(\alpha-1)$, $id: F(p, p\alpha-2, s-p(\alpha-1)) \mapsto \frakT_{s, p}^\infty(\mu)$ is bounded if and only if $\mu$ is an $\left[s+p(\alpha-1)\right]$-Carleson measure.   
\end{namedthm}

However, Statement \ref{20240627stat01} is a special case of Theorem \ref{20240624thm10}, which asserts under the same assumptions, $F(p, p\alpha-2, t) \hookrightarrow \frakT_{s, p}^\infty(\mu)$ is continuous if and only if $\mu$ is an $[s+p(\alpha-1)]$-Carleson measure for any $t<s$. Therefore, near-endpoints Carleson embedding is more efficient compared to Xiao-Pau-Zhao's analog when $\alpha>1$. We summarize as follows.

\begin{center}
\begin{tabular}{|p{4cm}|p{5.5cm}|p{5.5cm}|}
\hline
\multicolumn{3}{|c|}{}\\
\multicolumn{3}{|c|}{Near-Endpoints Carleson embedding V.S. Xiao-Pau-Zhao's analog for $\alpha>1$} \\ 
\multicolumn{3}{|c|}{}\\
\hline
&Behavior of $\mu$ & Continuous embedding\\
\hline
& &\\
Near-Endpoints Carleson embedding & $s+p(\alpha-1)$-Carleson measure & $F(p, p\alpha-2, t) \hookrightarrow \frakT_{s, p}^\infty(\mu)$ for any $t<s$   \\
& &\\
Theorem \ref{20240624thm10} & $\left\|\mu \right\|_{\mathcal{CM}_{s+p(\alpha-1)}}:=\sup\limits_{I \subset \partial \D} \frac{\mu(S(I))}{|I|^{s+p(\alpha-1)}}$ & $\left\|f \right\|^p_{\frakT_{s, p}^\infty}=\sup\limits_{I \subset \partial \D} \frac{\int_{S(I)}|f(z)|^pd\mu(z)}{|I|^s}$ \\
& &\\
\hline
& & \\
Xiao-Pau-Zhao's analog for $\alpha>1$ & $s+p(\alpha-1)$-Carleson measure & $F(p, p-2, s-p(\alpha-1)) \hookrightarrow \frakT^\infty_{s, p}(\mu)$ \\
& &\\
Statement \ref{20240627stat01} & $\left\|\mu \right\|_{\mathcal{CM}_{s+p(\alpha-1)}}=\sup\limits_{I \subseteq \partial \D}  \frac{\mu(S(I))}{|I|^{s+p(\alpha-1)}}$ & $\left\|f \right\|^p_{\frakT_{s, p}^\infty}=\sup\limits_{I \subseteq \partial \D} \frac{\int_{S(I)} |f(z)|^p d\mu(z)}{|I|^s}$  \\
& & \\
\hline
\end{tabular}
\\
\vspace{0.08in}
Table II: Near-endpoints Carleson embeddings for $1<\alpha<1+\frac{1-s}{p}$ if $p>2$ and $1<\alpha \le 1+\frac{1-s}{p}$ if $1<p \le 2$.
\end{center} 

\subsection{Case III: when $\alpha>1+\frac{1-s}{p}$  for $p>2$, and when $\alpha \ge 1+\frac{1-s}{p}$ for $1<p \le 2$.}

In this case, the continuous embedding $\calD_{s+p\alpha-2}^p \hookrightarrow L^p(\mu)$ is valid. Therefore, the endpoint $t=s$ can be achieved, and hence, in this situation, the near-endpoints Carleson embedding is not as good as the Carleson embedding induced by the continuous embedding of weighted Dirichlet spaces. 

\begin{center}
\begin{tabular}{|p{4.3cm}|p{5.5cm}|p{5.5cm}|}
\hline
\multicolumn{3}{|c|}{}\\
\multicolumn{3}{|c|}{Near-Endpoints Carleson embedding V.S. $\calD_{s+p\alpha-2}^p \hookrightarrow L^p(\mu)$} \\ 
\multicolumn{3}{|c|}{}\\
\hline
&Behavior of $\mu$ & Continuous embedding\\
\hline
& &\\
Near-Endpoints Carleson embedding & $s+p(\alpha-1)$-Carleson measure & $F(p, p\alpha-2, t) \hookrightarrow \frakT_{s, p}^\infty(\mu)$ for any $t<s$   \\
& & \\
Theorem \ref{20240624thm10} & $\left\|\mu \right\|_{\mathcal{CM}_{s+p(\alpha-1)}}=\sup\limits_{I \subset \partial \D} \frac{\mu(S(I))}{|I|^{s+p(\alpha-1)}}$ & $\left\|f \right\|^p_{\frakT_{s, p}^\infty}=\sup\limits_{I \subset \partial \D} \frac{\int_{S(I)}|f(z)|^pd\mu(z)}{|I|^s}$ \\
& &\\
\hline
& & \\
Carleson embedding induced by $\calD_{s+p\alpha-2}^p \hookrightarrow L^p(\mu)$ & $s+p(\alpha-1)$-Carleson measure & $F(p, p\alpha-2, s) \hookrightarrow \frakT^\infty_{s, p}(\mu)$ \\
& & \\
\cite[Theorem 4.4]{PauZhao2014} & $\left\|\mu \right\|_{\mathcal{CM}_{s+p(\alpha-1)}}=\sup\limits_{I \subseteq \partial \D}  \frac{\mu(S(I))}{|I|^{s+p(\alpha-1)}}$ & $\left\|f \right\|^p_{\frakT_{s, p}^\infty}=\sup\limits_{I \subseteq \partial \D} \frac{\int_{S(I)} |f(z)|^p d\mu(z)}{|I|^s}$  \\
& & \\
\hline
\end{tabular}
\\
\vspace{0.08in}
Table III: Near-endpoints Carleson embeddings for $\alpha>1+\frac{1-s}{p}$ if $p>2$ and $\alpha  \ge 1+\frac{1-s}{p}$ if $1<p \le 2$.
\end{center}

\end{document}